\documentclass[11pt]{amsart}

\usepackage[francais, english]{babel}
\usepackage[latin1]{inputenc}
\usepackage[T1]{fontenc}
\usepackage{amsthm,amsmath,amsfonts, amssymb,bbm}
\usepackage{hyperref,enumitem,stmaryrd}
\usepackage{graphicx,eepic,pstricks}
\usepackage{verbatim}
\def \RR{\mathbb R}

\def \SS{\mathbb S}
\def \CC{\mathbb C}

\def \F{\mathcal F}

\def \sclr#1#2{\langle #1,#2\rangle}

\def \eid{Ei^{(\delta)}}
\def \ei{Ei}
\def \gd{G^{(\delta)}}
\def \hd{H^{(\delta)}}

\renewcommand{\geq}{\geqslant}
\renewcommand{\leq}{\leqslant}

\usepackage{color}
\definecolor{darkgreen}{rgb}{0,0.4,0}
\definecolor{MyDarkBlue}{rgb}{0,0.08,0.50}
\definecolor{BrickRed}{rgb}{0.65,0.08,0}

\hypersetup{
colorlinks=true,       % false: boxed links; true: colored links
    linkcolor=blue,          % color of internal links
    citecolor=red,        % color of links to bibliography
    filecolor=BrickRed,      % color of file links
    urlcolor=darkgreen        % color of external links
}

\theoremstyle{plain}
\newtheorem{theorem}{Theorem}%[section]
\newtheorem*{theorem-sum}{Theorem}
\newtheorem{proposition}{Proposition}

\newtheorem{lemma}[proposition]{Lemma}

\theoremstyle{definition}

\title[]{A family of Fourier transform's eigenfunctions}

\author[Rodolphe Garbit]{Rodolphe Garbit$^{1}$}
\address{$^{1}$ Univ Angers, CNRS, LAREMA, SFR MATHSTIC, F-49000 Angers, France}
\email{rodolphe.garbit@univ-angers.fr}
\author[Julien-Bilal Zinoune]{Julien-Bilal Zinoune$^{2,3}$}
\address{$^{2}$ Univ Angers, LPHIA, SFR MATRIX, F-49000 Angers, France}
\address{$^{3}$ ESAIP, CERADE, 18 rue du 8 mai 1945, 49180 St-Barth\'elemy d'Anjou Cedex, France}
\email{julienbilal.zinoune@univ-angers.fr}

\thanks{The first author was supported by the ANR project ``Rawabranch'' number ANR-23-CE40-0008. The second author was supported by a doctoral contract awarded by the Universit\'e d'Angers and supplemented by financial support from ESAIP%, covering half of the total expenses.
}


\date{\today}

\begin{document}

\begin{abstract}
This paper presents a family of Fourier eigenfunctions indexed by the space dimension $d$. These eigenfunctions are radial and built upon some generalized exponential integral function. For $d=1,2, 3$, they are integrable or square integrable and give new explicit examples of Fourier eigenfunctions in the usual or Fourier-Plancherel sense. For $d\geq 4$, the functions are examples of non standard eigenfunctions, i.e.~eigenfunctions in the sense of distribution. The discovery of these eigenfunctions stems from research in thermal lens spectroscopy, at the intersection of thermodynamics and optics.
Their use could simplify the analysis of thermo-optical systems, paving the way for applications in optical computing, material studies and thermodynamic.
\end{abstract}

\maketitle

\section{Introduction and main result}

This paper deals with eigenfunctions of the Fourier transform, i.e.~functions $f$ whose Fourier transform $\widehat{f}$ satisfies $\widehat{f}=\lambda f$ for some complex number $\lambda$. Here, the Fourier transform of an integrable function $f:\RR^d\to \CC$, $d\geq 1$, is defined with the following convention:
\begin{equation}
\widehat{f}(t)=\int_{\RR^d} e^{-i\sclr{t}{x}} f(x) dx,
\end{equation}
where $t=(t_1,t_2,\ldots, t_d)$ and $x=(x_1,x_2,\ldots, x_d)$ belong to $\RR^d$, $\sclr{t}{x}=\sum_{i=1}^d t_ix_i$ is the standard inner product, and $dx$ stands for the $d$-dimensional Lebesgue measure.
We shall also write $\F(f)$ instead of $\widehat{f}$ at some points in the text.

A fundamental example of eigenfunction is the Gaussian function $\exp\left(-\Vert x \Vert^2/2 \right)$, where $\Vert x\Vert^2=x_1^2+x_2^2+\cdots+x_d^2$. This eigenfunction is associated with the eigenvalue $\lambda=(2\pi)^{d/2}$. 
In dimension $1$, other examples include the hyperbolic secant function, Hermite functions, or
the function $\vert x\vert^{-1/2}$. This last example is a \emph{non standard} eigenfunction, meaning the function is not integrable, nor square integrable; its Fourier transform is understood in the sense of \emph{distribution}. Other examples and formulas that generate eigenfunctions can be found in the book \cite[Chapter IX]{Tit48}.
In dimension $d\geq 2$, eigenfunctions can be obtainted by considering tensor products 
$f(x)=\prod_{k=1}^d f_i(x_i)$ where each $f_i$ is a one-dimensional eigenfunction.
Examples of eigenfunctions that are not (combinations of) tensor products of one-dimensional eigenfunctions include functions $P(x)\exp(-\Vert x\Vert^2/2)$, where $P$ is a homogeneous harmonic polynomial (see \cite[p.85]{Duo01}), or the radial function $\Vert x\Vert^{-d/2}$, which is a non standard one (see \cite[p.71]{Duo01}). In \cite{LaMa17}, the authors present another example of distributional eigenfunction of the planar Fourier transform, which is not a tensor product, namely the function 
$g(x)=\Vert x\Vert / x_1 x_2$. In \cite{LaMa18}, the same authors provide a method that generates planar eigenfunctions from another ones; they illustrate this with examples derived from the function $g$. All these examples have the form $R(x)\Vert x\Vert$, where $R=P/Q$ is some rational fraction and $P,Q$ are homogeneous symmetric or antisymmetric polynomials.

In this paper, we present a collection of explicit eigenfunctions indexed by dimension~$d$. 
For $d\geq 2$, those functions are not tensor products of one-dimensional eigenfunctions; they are radial functions. For $d\geq 4$, they are \emph{non standard} eigenfunctions.

The planar case $d=2$ was discovered by the second author while performing computations on thermal lensing spectroscopy, a technique used to study the absorption properties of species where the interaction of a laser beam with a medium generates a temperature gradient that modifies the phase of the beam. The study of the analytical model led to the discovery of a specific eigenfunction associated with this thermo-optical phenomenon \cite{zinoune2024analytical}. Details about the physical context are given in Section~\ref{physmotiv}.
Starting from this $2$-dimensional example, it is easy to infer a formula that works in any dimension $d\geq 1$. Let us present the general formula.

Let $\delta<2$ be a real number. We denote by $Ei^{(\delta)}$ the $\delta$-exponential function, which we define for $x>0$ by
$$\eid(x)=\int_{-\infty}^x \frac{e^t}{t^\delta} dt,$$
where $t^{\delta}$ means $-\vert t \vert^{\delta}$ for negative $t$. 
When $\delta=1$, this function is known as the \emph{exponential integral} function, and we shall write $\ei$ instead of $Ei^{(1)}$ for simplicity.
For $\delta<1$, the function under the integral symbol is integrable, but for $\delta \in [1,2)$ this is not the case and the formula must be interpreted as a principal value, meaning that
\begin{equation}
\label{eq:defexpintegrale}
\eid(x)=\lim_{\epsilon\to 0} \left\{\int_{-\infty}^{-\epsilon} +\int_{\epsilon}^x\right\} \frac{e^t}{t^{\delta}} dt.
\end{equation}
Using the $\delta$-exponential function, we construct a $d$-dimensional Fourier eigenfunction in the following way: To each integer $d\geq 1$ we associate the radial function $\phi_d : \RR^d \setminus\{0\} \to \RR$ which is defined for all $x=(x_1,x_2,\ldots, x_d)\in \RR^d \setminus\{0\}$ by
\begin{equation}
\label{eq:definitionphid} 
\phi_d(x)= r^{2-d}e^{-r^2/2}\eid(r^2/2), 
\end{equation}
where $r^2=\Vert x\Vert^2=x_1^2+x_2^2+\cdots + x_d^2$, and 
\begin{equation}
\label{eq:defdelta}
\delta=2-\frac{d}{2}.
\end{equation}
Integrability of function $\phi_d$ depends on the dimension $d$. In Section~\ref{sec:integrability} we shall prove the following:
\begin{itemize}
\item $\phi_1$ is integrable, hence it has a Fourier transform in the usual sense;
\item $\phi_2$ and $\phi_3$ are square integrable, thus having a Fourier transform in the sense of Fourier-Plancherel;
\item for $d\geq 4$, the functions $\phi_d$ do not belong to any $L^p$-space. They define \emph{tempered distributions} and their Fourier transform is understood in that sense.
\end{itemize}
With this in mind, we can state our main result:

\begin{theorem}
\label{mainthm}
For any $d\geq 1$, the function $\phi_d$ is an eigenfunction of the Fourier transform on $\RR^d$, associated with eigenvalue $-(2\pi)^{d/2}$.
\end{theorem}

Note that when the dimension $d=4$, the number $\delta$ is equal to $0$ and our function simplifies into $\Vert x\Vert^{-2}$, which is a well-known distributional eigenfunction of the form $\Vert x\Vert^{-d/2}$.

The discovery of these eigenfunctions in a physical context and the interest regarding physics is explained in Section~\ref{physmotiv}. 
The last Section~\ref{sec:proofs} is devoted to the proof of Theorem~\ref{mainthm}.

\section{Physical motivation}
\label{physmotiv}
The discovery of these eigenfunctions stems from research in thermal lens spectroscopy, a field that intersects thermodynamics and optics. Over the past 50 years, numerous studies have explored this technique~\cite{snook1995thermal,jacinto2006thermal,franko2023recent} because it is theoretically an absolute method for measuring the absorption of species in solution~\cite{brannon1978absolute,bindhu1996measurement}. Additionally, it can determine the quantum yields of diffusion or fluorescence in materials, which is crucial for fields such as semiconductors, photovoltaics, photonics, and optoelectronics~\cite{brannon1978absolute,bindhu1996measurement,zinoune2023wavelength,boudebs2023thermal,cruz2010quantum}.

In thermal lens spectroscopy, a collimated laser beam with a Gaussian intensity profile is focused onto a solution using a lens. The solution absorbs part of the laser's energy, creating a slight temperature gradient along the beam's path. This gradient induces a refractive index change, which modifies the phase of the laser beam. This phase shift is then observed as an intensity variation after diffraction to infinity.

Over time, several analytical models have been developed to analyze these results \cite{whinnery1974laser,sheldon1982laser,carter1984comparison,shen1992model,jurgensen1995studies,cabrera2009thermal}. Most of these models focus on measuring the central intensity variation of the laser beam using small-aperture Z-scan techniques, a method widely used in optics \cite{vermeulen2023nonlinear}, particularly for characterizing nonlinear materials. However, one limitation is the relatively high noise level, as the measurement is taken at a single point (the beam's center).

To address this issue, an alternative approach is to measure the variation in the laser beam's diameter using a CCD sensor \cite{boudebs2013nonlinear}. This multi-point measurement improves the signal-to-noise ratio. In a previous study, we developed the analytical model for this system \cite{zinoune2024analytical}, which led to the discovery of the eigenfunction. This is evident in equations~(6) and~(11) from our earlier work on the analytical development of thermal lensing in an afocal setup with a Gaussian profile \cite{zinoune2024analytical}.
Indeed, these equations simplified without constants or irrelevant physical elements are respectively
\begin{equation}
\label{eq:asympten31}
E_{th}(r,t) = \exp\left(-\frac{r^2}{2}\right)\left[\ei(-r^2) - \ei\left(-\frac{r^2}{4t+1}\right)\right],
\end{equation}
and
\begin{equation}
\label{eq:asympten32}
E_s(\rho,t) = 2\exp(-2\rho^2)\left[\ei\left(\frac{4\rho^2}{3}\right) - \ei\left(\frac{4\rho^2}{4t+3}\right)\right],
\end{equation}
where $r = \sqrt{x^2 + y^2}$ is the radial position in the normal domain, and $\rho = \sqrt{u^2 + v^2}$ is the radial position in the spectral domain. Here, $t$ represents time, $E_{\text{th}}(r,t)$ is the field amplitude at the exit of the liquid, and $E_s(\rho,t)$ represents the spectra in the CCD sensor plane. In \cite{zinoune2024analytical}, we obtained the expression for $E_s(\rho,t)$ given in equation~\eqref{eq:asympten32} from the one for $E_{\text{th}}(r,t)$ given in~\eqref{eq:asympten31} using Fourier optics formalism \cite{goodman2005introduction}. These two equations are connected through a Fourier transform along $x$ and $y$ as follows:
\begin{equation}
\label{eq:asympten33}
    E_s(\rho,t) = \mathcal{F}\left[E_{\text{th}}(r,t)\right](\rho,t). 
\end{equation}
Upon comparing equations~\eqref{eq:asympten31} and \eqref{eq:asympten32}, it becomes clear that the fields $E_{\text{th}}(r,t)$ and $E_s(\rho,t)$ are similar, differing only by a set of constants. Given the relationship outlined in equation~\eqref{eq:asympten33}, this suggests that $f(r) = \exp(-r^2) \text{Ei}(r^2)$ is almost an eigenfunction of the planar Fourier transform. We will now present the \emph{formal} calculation, following the method used in our previous work \cite{zinoune2024analytical}, to show that
\begin{equation}
\label{eq:fourierfonctionpropre2d}
\mathcal{F}[f(r)](\rho) = -\pi f\left(\rho/2\right).
\end{equation}
It is then immediate that the function $f(r/\sqrt{2})$ is an eigenfunction of the planar Fourier transform, associated with eigenvalue $-2\pi$.

First, by performing a simple change of variable, the exponential integral term $\ei(r^2)$ can be turned into

\begin{equation*}
\ei(r^2)=\int_{-\infty}^1 \frac{\exp(tr^2)}{t} dt.
\end{equation*}
So, starting from the definition of $f(r)$ and expanding the exponential integral $\ei(r^2)$ into its integral form, we have
%\begin{equation}
%\label{eq:asympten34}
%   \mathcal{F}[f(r)](\rho) = \mathcal{F}[\exp(-r^2)\ei(r^2)](\rho). 
%\end{equation}
%Expanding the exponential integral $\ei(r^2)$ into its integral form:
\begin{equation*}
\label{eq:asympten35}
   \mathcal{F}[f(r)](\rho) = \mathcal{F}\left[\exp(-r^2)\int_{-\infty}^1 \frac{\exp(tr^2)}{t} dt \right](\rho).
\end{equation*}
By combining the arguments of the exponentials
%in Equation~\eqref{eq:asympten35}
and exchanging the integral over $t$ with the Fourier transform, we obtain:
\begin{equation*}
\label{eq:asympten36}
\mathcal{F}[f(r)](\rho) = \int_{-\infty}^1 \mathcal{F}\left[\frac{\exp[-(1-t)r^2]}{t}\right](\rho) dt.
\end{equation*}
Taking the Fourier transform of the Gaussian function in
%Eq.\eqref{eq:asympten36}
the equation above, we get
\begin{equation*}
\label{eq:asympten37}
\mathcal{F}[f(r)](\rho) = \pi \int_{-\infty}^1 \frac{\exp\left(-\frac{1}{4} \frac{\rho^2}{1-t}\right)}{(1-t)t} dt.
\end{equation*}  
Then we use the change of variable $s = -\frac{t}{1-t}$, which gives
\begin{equation*}
\label{eq:asympten38}
\mathcal{F}[f(r)](\rho) =  -\pi \exp\left(-\frac{\rho^2}{4}\right) \int_{-\infty}^{1} \frac{\exp\left(\frac{\rho^2 s}{4}\right)}{s} ds.
\end{equation*}
Recognizing the last integral as an exponential integral function, we obtain:
\begin{equation*}
\label{eq:asympten39}
\mathcal{F}[f(r)](\rho) = -\pi \exp\left(-\frac{\rho^2}{4}\right)\text{Ei}\left(\frac{\rho^2}{4}\right).
\end{equation*}
Finally, identifying $f(\rho/2)$ on the right-hand side of that last equation leads to equation~\eqref{eq:fourierfonctionpropre2d}. Of course, these calculations deserve rigorous justification. This article aims to provide it.

\section{Proofs}
\label{sec:proofs}

This section is devoted to the proof of Theorem~\ref{mainthm}. First, we shall study the behavior of the $\delta$-exponential function at infinity and near the origin (Section~\ref{sec:estimatesonei}). Our estimates are then used to obtain integrability properties of the eigenfunctions (Section~\ref{sec:integrability}). We finally compute in Section~\ref{sec:compfouriertrans} the Fourier transform of our function $\phi_d$ (in fact, we shall work with a slightly modified version $f_d$ of it). 

\subsection{Modified exponential integral function estimates}
\label{sec:estimatesonei}

Let $\delta<2$ be a real number. Recall the definition of the $\delta$-exponential function:
$$\eid(x)=\int_{-\infty}^x \frac{e^t}{t^\delta} dt, \qquad x>0,$$
where $t^{\delta}$ means $-\vert t \vert^{\delta}$ for negative $t$. 
For $\delta \in [1,2[$ the function under the integral is not integrable and the formula must be interpreted as a principal value, meaning that
\begin{equation*}
\eid(x)=\lim_{\epsilon\to 0} \left\{\int_{-\infty}^{-\epsilon} +\int_{\epsilon}^x\right\} \frac{e^t}{t^{\delta}} dt.
\end{equation*}
To see that the limit exists, simply reverse time and cut the first integral into two pieces to obtain the following identity:
$$\left\{\int_{-\infty}^{-\epsilon} +\int_{\epsilon}^x\right\} \frac{e^t}{t^{\delta}} dt=\int_{\epsilon}^x \frac{2\sinh(t)}{t^{\delta}} dt - \int_{x}^{\infty}\frac{e^{-t}}{t^{\delta}}dt.$$
Since $\sinh(t)\sim t$ as $t\to 0$, the function $\sinh(t)/t^{\delta}$ is integrable on the interval $[0,x]$ as long as $\delta<2$. Hence, the $\delta$-exponential function can be written as
$$
\eid(x)= \int_{0}^x \frac{2\sinh(t)}{t^{\delta}} dt-\int_{x}^{\infty}\frac{e^{-t}}{t^{\delta}}dt.
$$
To take advantage of this formula, we set
$$G^{(\delta)}(x)=\int_{0}^x \frac{2\sinh(t)}{t^{\delta}} dt\quad\mbox{ and }\quad H^{(\delta)}(x)=\int_{x}^{\infty}\frac{e^{-t}}{t^{\delta}}dt,$$
so that
\begin{equation}
\label{eq:defexpintegralesansvp}
\eid(x)=\gd(x)-\hd(x).
\end{equation}

The behavior of function $\eid$ near zero and infinity is obtained in Sections \ref{sec:comportement_zero} and \ref{sec:comportement_infini} below through this decomposition.

\subsubsection{Behavior near zero}
\label{sec:comportement_zero}

Since $\gd(x)$ tends to $0$ as $x\to 0$, whereas $\hd(x)$ has a positive finite or infinite limit, we have that
$$\eid(x)\sim -\hd(x), \qquad x\to 0.$$
Now, for the asymptotics of $\hd(x)$, there are 3 cases.

First, if $\delta<1$ then the integral
$\int_{0}^{\infty}\frac{e^{-t}}{t^{\delta}}dt$
is finite, so that $\hd$ is continuous at $0$.

The second case is $\delta=1$. An integration by parts then gives
$$ H^{(1)}(x)=-\ln x\times  e^{-x} + \int_{x}^{\infty} \ln t \times  e^{-t} dt.$$
Since the function $\ln t \times  e^{-t}$ is integrable on $[0,\infty[$, we obtain that
$$H^{(1)}(x)=-\ln x\times  e^{-x} + \int_{0}^{\infty} \ln t \times  e^{-t} dt + o(1).$$
It is well known that $\int_{0}^{\infty} \ln t \times  e^{-t} dt=-\gamma$, where $\gamma$ is Euler's constant. Therefore,
$$H^{(1)}(x)=-\ln x -\gamma + o(1).$$

Finally, for $1<\delta<2$, an integration by parts gives the identity
$$\hd(x)=\frac{e^{-x}}{(\delta-1) x^{\delta -1}} -\frac{1}{\delta-1} \int_x^\infty \frac{e^{-t}}{t^{\delta-1}}dt,$$
from which we deduce 
$$\hd(x)\sim\frac{1}{(\delta-1) x^{\delta -1}} , \qquad x\to 0,$$
since the integral $\int_{0}^{\infty}\frac{e^{-t}}{t^{\delta-1}}dt$ is finite.

To summarize, as $x \to 0$, we have
\begin{equation}
\label{eq:asympten0}
\eid(x) \sim  -\hd(x) \sim
\begin{cases}
-\hd(0)<0 & \mbox{ if } \delta <1, \\
\ln x & \mbox{ if } \delta =1, \\
-\frac{1}{(\delta-1) x^{\delta -1}} & \mbox{ if } 1<\delta <2. \\
\end{cases}
\end{equation}

\subsubsection{Behavior at infinity}
\label{sec:comportement_infini}
As $x$ goes to infinity, the function $\hd(x)$ tends to zero, whereas $\gd(x)$ goes to infinity. More precisely, we have
%More precisely,  
%$$\hd(x)\sim \frac{e^{-x}}{x^{\delta}}, \quad x\to\infty.$$
%To see this, simply write
%$$
%x^{\delta}e^x \hd(x)= \int_x^{\infty}\left(\frac{x}{t}\right)^{\delta} e^{-(t-x)} dt = %\int_0^{\infty}\left(1+\frac{u}{x}\right)^{-\delta} e^{-u} du,
%$$
%and observe that the last integral converges to $\int_0^{\infty} e^{-u} du=1$ as %%$x\to\infty$ by the dominated convergence theorem. (The function under the integral symbol is dominated on $[0,+\infty[$ and uniformly in $x\geq 1$ by the function $(1+u)^{\vert \delta \vert}e^{-u}$, which is integrable.)
$$\gd(x)\sim \frac{e^{x}}{x^{\delta}}, \quad x\to\infty.$$
Indeed, since $2\sinh(t)/t^{\delta}\sim e^t/t^{\delta}$ as $t\to \infty$, the integral that defines $\gd(x)$ goes to $\infty$ as $x\to \infty$, therefore
$$
\gd(x)\sim \int_{1}^x \frac{2\sinh(t)}{t^{\delta}} dt\sim \int_1^x \frac{e^t}{t^{\delta}} dt.
$$
Now, the behavior of the last integral is obtained by writing
$$
x^{\delta}e^{-x}\int_1^x \frac{e^t}{t^{\delta}} dt=\int_1^x \left(\frac{x}{t}\right)^{\delta} e^{-(x-t)} dt = \int_0^{x-1}\left(1+\frac{u}{x-u}\right)^{\delta} e^{-u} du.
$$
As $x\to\infty$, the last integral converges to $\int_0^{\infty} e^{-u} du=1$ by the dominated convergence theorem. (The function under the integral symbol times the indicator function of $[0,x-1]$ is dominated on $[0,+\infty[$ and uniformly in $x\geq 1$ by the function $(1+u)^{\vert \delta \vert}e^{-u}$, which is integrable.)
So finally, we obtain
\begin{equation}
\label{eq:asymptalinfini}
\eid(x)\sim \frac{e^x}{x^{\delta}}, \qquad x\to \infty.
\end{equation}

\subsubsection{Back to the definition}

The definition of $\eid$ as a principal value is quite rigid and, at some point, we shall need a bit of flexibilty. This is the subject of the following lemma.
\begin{lemma}
\label{lem:eiddefwithdifferentspeed}
Let $a$ and $b$ be two positive functions of the real variable $\epsilon$ such that $a, b\to 0$ and $\vert a-b \vert= O(\min^2(a,b))$ as $\epsilon\to 0$. Then
$$\lim_{\epsilon\to 0} \left\{\int_{-\infty}^{-a} +\int_{b}^x\right\} \frac{e^t}{t^{\delta}} dt= \eid(x).$$
\end{lemma}
\begin{proof}
Let $c=\max(a,b)$ and write the above integral as
$$\left\{\int_{-\infty}^{-c} +\int_{c}^x\right\} \frac{e^t}{t^{\delta}} dt+ \int_{-c}^{-a} \frac{e^t}{t^{\delta}} dt +\int_{b}^c  \frac{e^t}{t^{\delta}} dt.$$
The first part converges to $\eid(x)$ as $c\to 0$, therefore it suffices to show that the two last integrals converge to zero. So we consider
$\int_{b}^c  \frac{e^t}{t^{\delta}} dt$ (the other one can be handled in the exact same way) and we focus on the case $1\leq\delta<2$ (else the result is obvious since the function is integrable and $c-b$ converges to zero). Since $e^t$ is increasing and $t^{\delta}$ is non decreasing, the integral can be bounded as follows:
$$\int_{b}^c  \frac{e^t}{t^{\delta}} dt\leq e^c \times \frac{c-b}{b^{\delta}}.$$
By assumption, there is a positive constant constant $K$ such that $\vert c-b \vert \leq K b^2$, hence $$\int_{b}^c  \frac{e^t}{t^{\delta}} dt \leq K b^{2-\delta},$$ and this proves the result since $\delta<2$.
\end{proof}

\subsection{Integrability of the eigenfunctions}
\label{sec:integrability}
Instead of the function $\phi_d$ defined in \eqref{eq:definitionphid} , we will work with a slighlty modified version of it. Namely, to each integer $d\geq 1$, we associate the radial function $f_d : \RR^d \setminus\{0\} \to \RR$ which is defined for all $x=(x_1,x_2,\ldots, x_d)\in \RR^d \setminus\{0\}$ by
\begin{equation}
\label{eq:definitionfd}
f_d(x)= r^{2-d}e^{-r^2}\eid(r^2), 
\end{equation}
where $r=\Vert x\Vert$ and $\delta=2-\frac{d}{2}$.
As $d\geq 1$ runs through the integers, the number $\delta$ takes the values $\frac{3}{2}$, $1$, $\frac{1}{2}$, $0$, $-\frac{1}{2}$, \dots Note that our eigenfunction $\phi_d$ is related to the function $f_d$ by the relation 
\begin{equation}
\label{eq:lienphidfd}
\phi_d(x)=(\sqrt{2})^{d-2}f_d(x/\sqrt{2}),
\end{equation}
thus they share the same integrability properties.

Let us now study the integrability of the function $f_d$. 
Because of \eqref{eq:asympten0}, as $r\to 0$, the following estimates hold:
\begin{equation}
\label{eq:fdenzero}
f_d(x) \sim
\begin{cases}
-r^{2-d}\hd(0) & \mbox{ if } d\geq 3, \\
2 \ln r & \mbox{ if } d=2, \\
-2 & \mbox{ if } d=1. \\
\end{cases}
\end{equation}
Now, as $r\to\infty$, it follows from \eqref{eq:asymptalinfini} that
\begin{equation}
\label{eq:fdalinfini}
f_d(x) \sim \frac{1}{r^2}.
\end{equation}
Let $p>0$ be given. Since $f_d$ is a radial function, it belongs to $L^p(\RR^d)$ if and only if 
\begin{align*}
\int_{0}^{\infty} \vert f_d(r)\vert^p r^{d-1} dr <\infty.
\end{align*}
Because of \eqref{eq:fdenzero}, the $\int_0^T$ part of the above integral is finite when $d=1$ or $d=2$ for any $p>0$, but for $d\geq 3$, it is finite if and only if $p<\frac{d}{d-2}$. On the other hand, estimate \eqref{eq:fdalinfini} ensures that the $\int_T^\infty$ part is finite if and only if $p>\frac{d}{2}$. Therefore, considering only integer values of $p$, we can state the following:
\begin{itemize}
\item the function $f_1$ belongs to $L^p(\RR)$ for all $p\geq 1$,
\item the function $f_2$ belongs to $L^p(\RR^2)$ if and only if $p\geq 2$,
\item the function $f_3$ belongs to $L^p(\RR^3)$ if and only if $p=2$,
\item for $d\geq 4$, the function $f_d$ does not belong to any $L^p(\RR^d)$.
\end{itemize}
In the last case, we shall use the theory of Fourier transform for tempered distributions in order to give a meaning to $\widehat{f_d}$. To see that function $f_d$ defines a tempered distribution, we first observe that the integral
$$C_d=\int_{\RR^d} \frac{\vert f_d(x)\vert}{1+\Vert x\Vert^{2d}} dx $$ 
is finite for any $d\geq 1$. This follows from \eqref{eq:fdenzero} and \eqref{eq:fdalinfini} above.
Now, for any test function $\phi$ that belongs to the Schwartz space $\mathcal{S}$ of infinitely differentiable functions rapidly decaying at infinity (see \cite[D\'efinition 9.2.1]{Bon06} for example), we have
\begin{equation}
\label{eq:fdisadistribution}
\vert \sclr{f_d}{\phi}\vert=\left\vert \int_{\RR^d} f_d(x) \phi(x) dx \right\vert \leq
C_d \sup_{x\in\RR^d}\vert (1+\Vert x\Vert^{2d}) \phi(x)\vert.
\end{equation}
Therefore, the function $f_d$ defines a continuous linear form on $\mathcal{S}$, i.e.~$f_d$ defines a tempered distribution (see \cite[Section 9.3]{Bon06}).

\subsection{Computation of the Fourier transform}
\label{sec:compfouriertrans}
Let $d\geq 1$ be fixed. 
Here again, for convenience, we will work with the function $f_d$ defined in \eqref{eq:definitionfd} instead of the eigenfunction $\phi_d$ defined in \eqref{eq:definitionphid}. Since $\phi_d(x)$ is equal to $f_d(x/\sqrt{2})$ up to a multiplicative constant (see equation~\eqref{eq:lienphidfd}), proving Theorem \ref{mainthm} is equivalent to establishing the relation
\begin{equation}
\label{eq:fourierfdrelation}
\widehat{f_d}(x)=-\pi^{d/2}f_d(x/2).
\end{equation}
Indeed, a simple calculation shows that if a function $f:\RR^d\to \CC$ satisfies the relation 
$$\widehat{f}(x)=\lambda f(\beta x)$$
for some $\lambda\in \CC$, $\beta>0$ and all $x\in \RR^d$, then the function $x\mapsto f(\sqrt{\beta}x)$ is an eigenfunction associated with eigenvalue $\lambda \beta^{-d/2}$.

To compute the Fourier transform of the function $f_d$, we introduce the functions $f_d^{\alpha}$, $\alpha>0$, which are defined by 
$$f_d^{\alpha}(x)=f_d(x)e^{-\alpha r^2} = r^{2-d}e^{-(1+\alpha) r^2}\eid(r^2),$$
where $r=\Vert x\Vert$, and 
$\delta=2-\frac{d}{2}$. These functions belong to $L^1(\RR^d)$ (see estimates \eqref{eq:fdenzero} and \eqref{eq:fdalinfini}). We shall compute their Fourier transform $\widehat{f_d^{\alpha}}$ and then take the limit as $\alpha\to 0$ to obtain the Fourier transform $\widehat{f_d}$ of $f_d$.

Thanks to decomposition \eqref{eq:defexpintegralesansvp} we can write
$$f_d^{\alpha}=g_d^{\alpha}-h_d^{\alpha},$$
where
$$g_d^{\alpha}(x)=r^{2-d}e^{-(1+\alpha)r^2} \gd(r^2) = e^{-(1+\alpha)r^2} \int_{0}^1 \frac{e^{r^2t}-e^{-r^2t}}{t^{\delta}} dt$$
and
$$h_d^{\alpha}(x)=r^{2-d}e^{-(1+\alpha)r^2} \hd(r^2)= e^{-(1+\alpha)r^2} \int_{1}^{\infty}\frac{e^{-r^2t}}{t^{\delta}}dt.$$
(Note that the last integral form for each function is obtained through the change of variable $t\leftarrow r^2 t$.)
The estimates of Sections \ref{sec:comportement_zero} and \ref{sec:comportement_infini} show that these two functions belong to $L^1(\RR^d)$, hence
\begin{equation}
\label{eq:decomp_fourier_f_alpha}
\widehat{f_d^{\alpha}}(u)=\widehat{g_d^{\alpha}}(u)-\widehat{h_d^{\alpha}}(u),
\end{equation}
for all $u=(u_1,u_2,\ldots,u_d)\in \RR^d$.
We first compute the Fourier transform of $\widehat{h_d^{\alpha}}$ of $h_d^{\alpha}$. By Fubini's theorem, we have
\begin{align}
\label{eq:fubini_h_alpha}
\widehat{h_d^{\alpha}}(u) &= \int_{\RR^d} e^{-i \sclr{u}{x}} e^{-(1+\alpha)r^2}  \left(\int_{1}^{\infty} \frac{e^{-r^2t}}{t^{\delta}} dt\right) dx \\
	\nonumber			& = \int_{1}^{\infty} \left( \int_{\RR^d} e^{-i \sclr{u}{x}} e^{-(1+\alpha+t)r^2 }dx \right)  \frac{dt}{t^{\delta}}.
\end{align}
(See Section \ref{sec:fubini_h_alpha} for the verification of Fubini's theorem hypotheses here.)
The integral inside the parentheses is the Fourier transform of the $d$-dimensional Gaussian function
$e^{-(1+\alpha+t) (x_1^2+x_2^2+\cdots +x_d^2)}$. It is equal to
$$\left(\frac{\pi}{1+\alpha+t}\right)^{d/2}\exp\left( -\frac{\rho^2}{4(1+\alpha+t)}\right),$$
where $\rho^2=u_1^2+u_2^2+\cdots+u_d^2$. Therefore
\begin{equation}
\label{eq:lastexpressionhalpha}
\widehat{h_d^{\alpha}}(u) = \pi^{d/2} \int_{1}^{\infty} \frac{1}{(1+\alpha+t)^{d/2} t^{\delta}} \exp\left( -\frac{\rho^2}{4(1+\alpha+t)}\right) dt.
\end{equation}
By the dominated convergence theorem (see Section \ref{sec:conv_dom_h_alpha} for the details), we get
\begin{align}
\label{eq:conv_dom_h_alpha}
\lim_{\alpha\to 0} \widehat{h_d^{\alpha}}(u) & = \pi^{d/2} \int_{1}^{\infty} \frac{1}{(1+t)^{d/2} t^{\delta}} \exp\left( -\frac{\rho^2}{4(1+t)}\right) dt\\
\nonumber	& = \pi^{d/2} e^{-\rho^2/4} \int_{1}^{\infty} \frac{1}{(1+t)^{d/2} t^{\delta}} \exp\left(\frac{\rho^2t}{4(1+t)}\right) dt.
\end{align}
Now we perform the change of variable $s=\dfrac{t}{1+t}$. Since $1+t=\dfrac{1}{1-s}$, $t=\dfrac{s}{1-s}$ and $dt=\dfrac{ds}{(1-s)^2}$, we see that
$$ \frac{1}{(1+t)^{d/2} t^{\delta}} dt = \frac{(1-s)^{d/2}(1-s)^\delta}{s^{\delta}} \dfrac{ds}{(1-s)^2}=\frac{ds}{s^\delta},$$
where we have used the relation $\delta=2-\frac{d}{2}$. It follows that
\begin{equation}
\label{eq:limite_1}
\lim_{\alpha\to 0} \widehat{h_d^{\alpha}}(u) = \pi^{d/2} e^{-\rho^2/4} \int_{1/2}^{1} \exp\left(\frac{\rho^2 s}{4}\right)\frac{ds}{s^{\delta}}.
\end{equation}

Let us now compute the Fourier transform $\widehat{g_d^{\alpha}}$. As in the computation of $\widehat{h_d^{\alpha}}$, we use Fubini's theorem (see Section \ref{sec:fubini_g_alpha} for the justification) and the explicit expression of the Fourier transform of a Gaussian to obtain 
\begin{align}
\label{eq:fubini_g_alpha}
\widehat{g_d^{\alpha}}(u) &= \int_{\RR^d} e^{-i\sclr{u}{x}} e^{-(1+\alpha)r^2}  \left(\int_{0}^1 \frac{e^{r^2t}-e^{-r^2t}}{t^{\delta}} dt\right) dx\\
\nonumber		&= \int_{0}^1 \left( \int_{\RR^d} e^{-i\sclr{u}{x}} \left( e^{-(1+\alpha-t)r^2}-e^{-(1+\alpha+t)r^2} \right) dx \right) \frac{dt}{t^{\delta}}\\
\nonumber		&= \pi^{d/2} \int_{0}^1 \left( \frac{\exp\left( -\frac{\rho^2}{4(1+\alpha-t)}\right)}{(1+\alpha-t)^{d/2}}-\frac{\exp\left( -\frac{\rho^2}{4(1+\alpha+t)}\right)}{(1+\alpha+t)^{d/2}} \right) \frac{dt}{t^{\delta}}.
\end{align}
By the dominated convergence theorem (see Section \ref{sec:conv_dom_g_alpha} for the justification), we obtain
\begin{align}
\label{eq:conv_dom_g_alpha}
\lim_{\alpha\to 0} \widehat{g_d^{\alpha}}(u)  & =\pi^{d/2} \int_{0}^1 \left( \frac{\exp\left( -\frac{\rho^2}{4(1-t)}\right)}{(1-t)^{d/2}}-\frac{\exp\left( -\frac{\rho^2}{4(1+t)}\right)}{(1+t)^{d/2}} \right) \frac{dt}{t^{\delta}} \\
\nonumber & =\pi^{d/2} e^{-\rho^2/4} \int_{0}^1 \left( \frac{\exp\left( -\frac{\rho^2 t}{4(1-t)}\right)}{(1-t)^{d/2}}-\frac{\exp\left( \frac{\rho^2t}{4(1+t)}\right)}{(1+t)^{d/2}} \right) \frac{dt}{t^{\delta}}.
\end{align}
When $\delta=\frac{3}{2}$ or $1$ (i.e.~$d=1$ or $2$), we can not separate the integral above into the two ``obvious'' parts because the two functions are not separately integrable at $0$. For this reason, we write the integral as the limit of the integral over $[\epsilon, 1]$, so that 
\begin{equation}
\label{eq:limite_2}
\lim_{\alpha\to 0} \widehat{g_d^{\alpha}}(u) = \lim_{\epsilon \to 0} \pi^{d/2} e^{-\rho^2/4} (A_{\epsilon}-B_{\epsilon}),
\end{equation}
where
$$
A_{\epsilon} = \int_{\epsilon}^1 \frac{\exp\left( -\frac{\rho^2t}{4(1-t)}\right)}{(1-t)^{d/2} t^{\delta}} dt
\quad\mbox{ and }\quad
B_{\epsilon} = \int_{\epsilon}^1 \frac{\exp\left( \frac{\rho^2t}{4(1+t)}\right)}{(1+t)^{d/2} t^{\delta}} dt.
$$
The change of variables $s=\frac{-t}{1-t}$ in the first integral gives
$$A=-\int_{-\infty}^{\frac{-\epsilon}{1-\epsilon}} \exp\left( \frac{\rho^2s}{4}\right) \frac{ds}{s^{\delta}}.$$
In the second integral, we use the change of variables $s=\frac{t}{1+t}$ to obtain
$$B= \int_{\frac{\epsilon}{1+\epsilon}}^{1/2} \exp\left( \frac{\rho^2s}{4}\right)\frac{ds}{s^{\delta}}.$$
Therefore, adding the two parts \eqref{eq:limite_1} and \eqref{eq:limite_2}, we get that $\lim_{\alpha\to 0}\widehat{f^{\alpha}_d}(u)$ is the limit of
\begin{equation*}
-\pi^{d/2} e^{-\rho^2/4}  \left\{\int_{-\infty}^{\frac{-\epsilon}{1-\epsilon}} +\int_{\frac{\epsilon}{1+\epsilon}}^1\right\} \exp\left(\frac{\rho^2s}{4}\right)\frac{ds}{s^{\delta}}
\end{equation*}
 as $\epsilon\to 0$.
Therefore, it follows from Lemma~\ref{lem:eiddefwithdifferentspeed} that
\begin{equation}
\label{convergenceponctuelle}
\lim_{\alpha\to 0}\widehat{f^{\alpha}_d}(u)=-\pi^{d/2} e^{-\rho^2/4} \eid(\rho^2/4)=-\pi^{d/2} f_d(u/2),
\end{equation}
for all $u\in\RR^d$.
It remains to connect $\lim_{\alpha\to 0}\widehat{f^{\alpha}_d}$ and $\widehat{f_d}$ so as to prove that 
\begin{equation}
\label{finalequality}
\widehat{f_d}(u)=-\pi^{d/2} f_d(u/2)
\end{equation}
for almost all $u\in\RR^d$.

For $d=1,2,3$, this follows quite directly from the fact that $f_d$ belongs to $L^2(\RR^d)$. Indeed, since $\vert f_d^{\alpha}(x)\vert =\vert f_d(x) e^{-\alpha r^2}\vert \leq \vert f_d(x)\vert$, the dominated convergence theorem ensures that $f_d^{\alpha}\to f_d$ in $L^2(\RR^d)$, so that 
$\widehat{f^{\alpha}_d}\to \widehat{f_d}$ in $L^2(\RR^d)$ too. By extracting a subsequence that converges almost everywhere, we obtain the equality~\eqref{finalequality}.

In other cases ($d\geq 4$), the function $f_d$ doesn't belong to any $L^p(\RR^d)$ space, and we have to work in $\mathcal{S}'$. In this space, it holds that $\widehat{f_d^{\alpha}}\to \widehat{f_d}$.
Indeed, since the Fourier transformation is a continuous operator on $\mathcal{S}'$ endowed with the weak topology (see \cite[Th\'eor\`eme 9.4.3]{Bon06}), it suffices to show that $f_d^{\alpha}\to f_d$ in $\mathcal{S}'$, i.e. for all $\phi\in \mathcal{S}$, 
$$\lim_{\alpha\to 0} \sclr{f_d^{\alpha}}{\phi}=\sclr{f_d}{\phi}.$$
To see this, we use the same bound as in~\eqref{eq:fdisadistribution}:
\begin{equation*}
\vert \sclr{f_d^{\alpha}}{\phi}-\sclr{f_d}{\phi}\vert  \leq 
C_d^{\alpha} \sup_{x\in\RR^d}\vert (1+\Vert x\Vert^{2d}) \phi(x)\vert,
\end{equation*}
where 
\begin{equation*}
C_d^{\alpha}=\int_{\RR^d} \frac{\vert f_d^{\alpha}(x)-f_d(x)\vert}{1+\Vert x\Vert^{2d}} dx
=\int_{\RR^d}\frac{\vert f_d(x)\vert }{1+\Vert x\Vert^{2d}} (1-e^{-\alpha r^2}) dx.
\end{equation*}
Since the function $\dfrac{\vert f_d(x)\vert }{1+\Vert x\Vert^{2d}}$ is integrable, it follows from the dominated convergence theorem that $\lim_{\alpha\to 0} C_d^{\alpha}=0$. Therefore $f_d^{\alpha}\to f_d$ in $\mathcal{S}'$, and, by consequence,
$\widehat{f_d^{\alpha}}\to \widehat{f_d}$ in $\mathcal{S}'$ too.
Put $g(u)=-\pi^{d/2} f_d(u/2)$. We know from \eqref{convergenceponctuelle} that $\widehat{f_d^{\alpha}}(u)\to g(u)$ for every $u\in\RR^d$. In order to identify $\widehat{f_d}$ with $g$, we show that the convergence also holds in $\mathcal{S}'$. To do this, let us consider the function $\dfrac{\widehat{f_d^{\alpha}}(u)}{1+\Vert u\Vert^{2d+4}} $ which converges almost everywhere to $\dfrac{g(u)}{1+\Vert u\Vert^{2d+4}} $. Thanks to inequality \eqref{eq:boundonfdalphauniform} in Section~\ref{uniformboundfdalphahat} below, we have
$$\dfrac{\vert \widehat{f_d^{\alpha}}(u)\vert}{1+\Vert u\Vert^{2d+4}}\leq \frac{A\rho^{-2}+B \rho^4}{1+\rho^{2d+4}},$$
where $\rho=\Vert u\Vert$. Since the function on the right hand side of the inequality is integrable on $\RR^d$ for all $d\geq 3$, the dominated convergence theorem ensures that
$\dfrac{\widehat{f_d^{\alpha}}(u)}{1+\Vert u\Vert^{2d+4}} $ converges to $\dfrac{g(u)}{1+\Vert u\Vert^{2d+4}} $ in $L^1(\RR^d)$, hence in $\mathcal{S}'$ too. This implies that  $\widehat{f_d^{\alpha}}$ converges to $g$ in $\mathcal{S}'$ because the space of test functions $\mathcal{S}$ is stable by multiplication by a polynomial.

As a conclusion, the equality $\widehat{f_d}=g$ holds in $\mathcal{S}'$, hence the distribution $\widehat{f_d}$ is regular and equality~\eqref{finalequality} holds.

\subsection{Additional computations}

\subsubsection{Use of Fubini's theorem in equation \eqref{eq:fubini_h_alpha} }
\label{sec:fubini_h_alpha}
In order to justify the use of Fubini's theorem in \eqref{eq:fubini_h_alpha}, we have to check that the following integral is finite:
$$I=\int_{1}^{\infty} \left( \int_{\RR^d} e^{-(1+\alpha+t)r^2 }dx\right)  \frac{dt}{t^{\delta}}.$$
Integrating in polar coordinates gives
$$I = c\int_{1}^{\infty} \left( \int_{0}^{\infty} e^{-(1+\alpha+t)r^2 }r^{d-1} dr\right)  \frac{dt}{t^{\delta}},$$
where $c$ is the surface area of the unit sphere $\SS^{d-1}$. Using the linear change of variable $s=\sqrt{1+\alpha+t}\times r$ in the integral inside the parenthesis, we obtain
\begin{align*}
I & = c \int_{1}^{\infty} \left( \frac{1}{(1+\alpha+t)^{d/2}}\int_{0}^{\infty} e^{-s^2 }s^{d-1} ds\right)  \frac{dt}{t^{\delta}}\\
		& = c' \int_{1}^{\infty} \frac{dt}{(1+\alpha + t)^{d/2} t^{\delta}},
\end{align*}
where $c'$ is some positive constant. As $t\to\infty$, the denominator of the above fraction is asymptotically equivalent to $t^{d/2}\times t^\delta=t^2$, hence the integral is finite.

\subsubsection{Justification of the convergence in \eqref{eq:conv_dom_h_alpha}} 
\label{sec:conv_dom_h_alpha}
In order to use the dominated convergence theorem in \eqref{eq:conv_dom_h_alpha}, we have to bound uniformly w.r.t. $\alpha$ the function
$$t\mapsto\frac{1}{(1+\alpha+t)^{d/2}t^{\delta}} \exp\left( -\frac{\rho^2}{4(1+\alpha+t)}\right)$$
by an integrable function on $[1,+\infty[$. Here it suffices to consider the function
$$t\mapsto \frac{1}{(1+t)^{d/2}t^{\delta}},$$
%$$t\mapsto \frac{1}{(1+t)^{d/2}t^{\delta}} \exp\left( -\frac{\rho^2}{4(2+t)}\right),$$
which is integrable on $[1,+\infty[$ since $\frac{d}{2}+\delta=2$.

\subsubsection{Use of Fubini's theorem in equation \eqref{eq:fubini_g_alpha}}
\label{sec:fubini_g_alpha}
Here we have to check that the following integral is finite: 
$$J=\int_{0}^1 \left( \int_{\RR^d} \left( e^{-(1+\alpha-t)r^2}-e^{-(1+\alpha+t)r^2} \right) dx\right) \frac{dt}{t^{\delta}}.$$
Integrating in polar coordinates gives
$$
J = c\int_{0}^{1} \left( \int_{0}^{\infty} \left( e^{-(1+\alpha-t)r^2}-e^{-(1+\alpha+t)r^2} \right)r^{d-1} dr\right)  \frac{dt}{t^{\delta}},
$$
where $c$ is the surface area of the unit sphere $\SS^{d-1}$. Using the same argument as in Section~\ref{sec:fubini_h_alpha} we get
$$
J = c'\int_{0}^{1} \left( \frac{1}{(1+\alpha - t)^{d/2}}  -  \frac{1}{(1+\alpha +t)^{d/2} } \right) \frac{dt}{t^{\delta}},
$$
where $c'$ is some positive constant. The function inside the integral is continuous on $(0,1]$ and, as $t\to 0$, it is asymptotically equivalent to
$$\frac{d}{(1+\alpha)^{\frac{d}{2}+1}}\times \frac{1}{t^{\delta -1}}.$$
Since $\delta-1=1-\dfrac{d}{2}<\frac{1}{2}$, the integral is finite.
%(Note: here the worst case is when $d=1$.)

\subsubsection{Justification of the convergence in \eqref{eq:conv_dom_g_alpha}}
\label{sec:conv_dom_g_alpha}
First we use the change of variable $t=(1+\alpha)x$ in the last integral of equation \eqref{eq:fubini_g_alpha} to obtain
\begin{equation}
\label{eq:fubini_derniere_expression}
\widehat{g_d^{\alpha}}(u)=\frac{\pi^{d/2}}{1+\alpha}\int_{0}^{\frac{1}{1+\alpha}} \left( \frac{\exp\left( -\frac{A}{1-x}\right)}{(1-x)^{d/2}}-\frac{\exp\left( -\frac{A}{1+x}\right)}{(1+x)^{d/2}} \right) \frac{dx}{x^{\delta}},
\end{equation}
where $A=\rho^2/4(1+\alpha)$. We are going to bound on $[0,1]$ the function under the integral sign, uniformly for $\alpha\in [0,1]$, i.e.~as $A$ remains in $[\rho^2/8, \rho^2/4]$, by an integrable function. We shall use different bounds on the intervals $[0,1/2]$ and $[1/2,1]$.\footnote{It is not necessary if one only wants to apply the dominated convergence theorem, but we need more precise bounds on $\widehat{g_d^{\alpha}}$ in order to show convergence in the space of distributions.}

First, on $[1/2, 1]$, the very simple bound
\begin{equation}
\label{eq:boundonend}
\left\vert \frac{1}{x^{\delta}}\left ( \frac{\exp\left( -\frac{A}{1-x}\right)}{(1-x)^{d/2}}-\frac{\exp\left( -\frac{A}{1+x}\right)}{(1+x)^{d/2}}\right) \right \vert
  \leq \frac{1}{x^{\delta}}\left(\frac{\exp\left( -\frac{\rho^2}{8(1-x)}\right)}{(1-x)^{d/2}}+\frac{\exp\left( -\frac{\rho^2}{8(1+x)}\right)}{(1+x)^{d/2 }}\right),
\end{equation}
is sufficient since the function on the right is continuous on $[1/2,1]$.

An integrable uniform bound on $[0,1/2]$ is a bit harder to obtain. Let us notice that the function inside the integral \eqref{eq:fubini_derniere_expression} has the form
\begin{equation}
\label{eq:defpsiA}
\psi_A(x)=\frac{\phi_{A}(-x)-\phi_{A}(x)}{x^{\delta}},
\end{equation}
where
$$\phi_A(x)=\frac{\exp\left( -\frac{A}{1+x}\right)}{(1+x)^{d/2}}.$$
The last function is $C^2$ on $[-1,1]$ and its first two derivatives are given by
$$\phi_A'(x)=\left(A-\frac{d}{2}(1+x)\right)\frac{\exp\left( -\frac{A}{1+x}\right)}{(1+x)^{2+\frac{d}{2}}},$$
and
$$\phi_A''(x)=\left(A^2-(d+2)(1+x)A+\frac{1}{4}d(d+2)(1+x)^2\right)\frac{\exp\left( -\frac{A}{1+x}\right)}{(1+x)^{\frac{d}{2}+4}}.$$
The second derivative can be bounded on $[-1/2,1/2]$ by
$$(A^2+2(d+2)A+d(d+2)) \sup_{y\in[2/3,2]} y^{\gamma}e^{-Ay},$$
where $\gamma=\frac{d}{2}+4$.
Since the function $y\mapsto y^{\gamma}e^{-Ay}$ reaches its maximum value on $[0,\infty)$ at $y=\gamma/A$ and is non-decreasing until this value, the following upper bound holds:
$$\sup_{y\in[2/3,2]} y^{\gamma}e^{-Ay}\leq 2^{\gamma}e^{-2A} 1_{\{2A\leq \gamma\}}+\frac{\gamma^{\gamma}}{A^{\gamma}}e^{-\gamma} 1_{\{2A> \gamma\}}\leq C,$$
where $C$ is some constant only depending on $\gamma$, hence $d$.
To summarize, the second derivative $\phi_A''$ is bounded on $[-1/2,1/2]$ by some constant $P_A=bA^2+c$, where $b,c$ are positive constants only depending on $d$.
Thanks to Taylor-Lagrange inequality, for all $x\in [-1/2,1/2]$,
$$\vert\phi_A(x)-(\phi_A(0)+ \phi_A'(0)x)\vert\leq \frac{P_A\vert x\vert^2}{2!}.$$
Therefore, for every $x\in[0,1/2]$,
\begin{align}
\label{eq:boundgnearzero}
\nonumber\vert\psi_A(x)\vert & \leq 2\vert \phi_A'(0)\vert \vert x\vert^{1-\delta}+ P_A \vert x\vert^{2-\delta} \\
 & \leq \vert 2A-d\vert e^{-A} \vert x\vert^{1-\delta}+ P_A \vert x\vert^{2-\delta}\\
\nonumber & \leq C \vert x\vert^{1-\delta}+ P_A \vert x\vert^{2-\delta},
\end{align}
where $C$ is some absolute constant.
Since $\delta<2$, the function on the right side is integrable on $[0,1/2]$. This concludes the proof.

\subsubsection{A uniform bound for $\widehat{f_d^{\alpha}}$}
\label{uniformboundfdalphahat}

In this section, we show that there exist constants $A$ and $B$, only depending on the dimension $d$, such that, for all $\alpha\in [0,1]$, 
\begin{equation}
\label{eq:boundonfdalphauniform}
\vert \widehat{f_d^{\alpha}}(u)\vert\leq \frac{A}{\rho^2}+B \rho^4,
\end{equation}
where $\rho=\Vert u\Vert$.
Recall that $\widehat{f_d^{\alpha}}(u)=\widehat{g_d^{\alpha}}(u)-\widehat{h_d^{\alpha}}(u)$. First, we bound  $\widehat{g_d^{\alpha}}(u)$ using the bounds of Section~\ref{sec:conv_dom_g_alpha}. Recall from \eqref{eq:fubini_derniere_expression} that
$$ \widehat{g_d^{\alpha}}(u)=\frac{\pi^{d/2}}{1+\alpha}\int_{0}^{\frac{1}{1+\alpha}} \psi_A(x) dx $$
with $A=\rho^2/4$, and $\psi_A$ defined by equation~\eqref{eq:defpsiA}.

Integrating inequality~\eqref{eq:boundonend} over $[1/2,1]$, we obtain
\begin{align*}
\int_{1/2}^1 \vert \psi_A(x)\vert dx & \leq \int_{1/2}^1 \frac{\exp\left( -\frac{\rho^2}{8(1-x)}\right)}{x^{\delta}(1-x)^{d/2}} dx +\int_{1/2}^1 \frac{\exp\left( -\frac{\rho^2}{8(1+x)}\right)}{x^{\delta}(1+x)^{d/2}} dx\\
&\leq \int_{2}^\infty \frac{t^{\delta+\frac{d}{2}}}{(t-1)^{\delta}} e^{-\frac{t\rho^2}{8}} dt + \int_{1/2}^{2/3} \frac{t^{\delta+\frac{d}{2}}}{(1-t)^{\delta}} e^{-\frac{t\rho^2}{8}} dt \\
&\leq C\int_{2}^\infty  e^{-\frac{t\rho^2}{8}} dt + \int_{1/2}^{2/3} \frac{t^{\delta+\frac{d}{2}}}{(1-t)^{\delta}} dt \\
&\leq \frac{C'}{\rho^2}+C'',
\end{align*}
where $C, C', C''$ are some absolute constants. Now, integrating inequality~\eqref{eq:boundgnearzero} over $[0,1/2]$ gives
\begin{equation*}
\int_0^{1/2} \vert \psi_A(x)\vert dx \leq C'''(1+P_A)\leq b'\rho^4+c',
\end{equation*}
where $C''', b', c'$ are some constants depending only on $d$. Putting things together, we obtain the bound
$$\vert \widehat{g_d^{\alpha}}(u)\vert \leq  \frac{K}{\rho^2}+ K'\rho^4,$$
where $K$ and $K'$ are constants depending only on the dimension $d$.

It remains to bound $\vert \widehat{h_d^{\alpha}}(u)\vert$.
Using the expression \eqref{eq:lastexpressionhalpha}, we see that
$$\vert \widehat{h_d^{\alpha}}(u)\vert\leq \pi^{d/2} \int_{1}^{\infty} \frac{1}{(1+t)^{d/2} t^{\delta}} dt.$$
Since $\frac{d}{2}+\delta =2$, the last integral is finite, so that $\vert \widehat{h_d^{\alpha}}(u)\vert$ is bounded by an absolute constant $K''$, independently of $u$ and $\alpha\in[0,1]$. Hence, we have
$$\vert \widehat{f_d^{\alpha}}(u)\vert\leq\vert \widehat{g_d^{\alpha}}(u)\vert+\vert \widehat{h_d^{\alpha}}(u)\vert \leq \frac{K}{\rho^2}+ K'\rho^4 +K'',$$
and the bound \eqref{eq:boundonfdalphauniform} holds.

\section*{Acknowledgments}

The authors would like to thank Lo\"ic Chaumont for his role in having made this collaboration possible. They would also like to thank Georges Boudebs, Mihaela Chis, and Christophe Cassagne for introducing them to thermal lens spectroscopy and for their supervision during J.~B.~Zinoune's PhD thesis.

\bibliographystyle{unsrt}
\makeatletter
\def\@openbib@code{\itemsep=-3pt}
\makeatother

\end{document}